\documentclass[12pt]{amsart}
\usepackage{amssymb,amsmath,amsthm,latexsym}
\usepackage{mathrsfs}
\usepackage{a4wide}
\usepackage[usenames,dvipsnames]{color}

\def\a{\alpha}

\def\b{\beta}

\newtheorem{theorem}{Theorem}[section]
\newtheorem{lemma}[theorem]{Lemma}
\newtheorem{corollary}[theorem]{Corollary}
\newtheorem{proposition}[theorem]{Proposition}
\theoremstyle{remark}
\newtheorem{remark}[theorem]{Remark}
\newtheorem*{remark*}{Remark}
\theoremstyle{definition}

\numberwithin{equation}{section}
\makeatother

\newcounter{smallromans}

  {\end{list}}

\newcounter{smallalphs}

  {\end{list}}

\begin{document}
\address{Department of Mathematics, Texas A\&M University, College Station, TX  77843 U.S.A}
\email{johnson@math.tamu.edu}
\address{Institute of Mathematics, Polish Academy of Sciences, \'{S}niadeckich 8, 00-956 Warszawa, Poland \mbox{and} Department of Mathematics and Statistics, Fylde College, Lancaster University, Lancaster LA1 4YF, United Kingdom \mbox{and} School of Mathematical Sciences, Western Gateway Building, University College Cork, Cork, Ireland}
\email{tomasz.marcin.kania@gmail.com}
\address{Department of Mathematics, Weizmann Institute of Science, Rehovot, Israel}
\email{gideon@weizmann.ac.il}

\title[Banach spaces of functions with countable support]{Closed ideals of operators on and complemented subspaces of Banach spaces of functions\\ with countable support}
\subjclass[2010]%
{Primary 46H10,
% Ideals and subalgebras
47B38,
% Operators on function spaces
47L10;
% Algebras of operators on Banach spaces and other topological linear spaces
Secondary 06F30,
% Topological lattices, order topologies
46B26,
% Nonseparable Banach spaces
47L20}

\author{William B. Johnson, Tomasz Kania, and Gideon Schechtman}

\thanks{The first-named author was supported in part by NSF DMS-1301604  and the U.S.-Israel Binational Science Foundation, the second-named author was supported by the Warsaw Centre of Mathematics and Computer Science, and the third-named author was supported in part by the U.S.-Israel Binational Science Foundation. Part of this work was done  when the third-named author was participating in the NSF Workshop in Analysis and Probability held at Texas A\&M University}

\begin{abstract}
Let $\lambda$ be an infinite cardinal number and let $\ell_\infty^c(\lambda)$ denote the subspace of $\ell_\infty(\lambda)$ consisting of all functions that assume at most countably many non-zero values. We classify all infinite dimensional complemented subspaces of $\ell_\infty^c(\lambda)$, proving that they are isomorphic to $\ell_\infty^c(\kappa)$ for some cardinal number $\kappa$. Then we show that the Banach algebra of all bounded linear operators on $\ell_\infty^c(\lambda)$ or $\ell_\infty(\lambda)$ has the unique maximal ideal consisting of operators through which  the identity operator does not factor. Using similar techniques, we obtain an alternative to Daws' approach description of the lattice of all closed ideals of $\mathscr{B}(X)$, where $X = c_0(\lambda)$ or $X=\ell_p(\lambda)$ for some $p\in [1,\infty)$, and we classify the closed ideals of $\mathscr{B}(\ell_\infty^c(\lambda))$ that contains the ideal of weakly compact operators.\end{abstract}

\maketitle

\section{Introduction and the statements of the main results}
The aim of this paper is to contribute to the two closely related programs of classifying complemented subspaces of Banach spaces and classifying the  maximal ideals of (bounded, linear) operators acting thereon. Our results constitute a natural continuation of the research undertaken by the authors in \cite{djs, kanlaus3, kanlaus2, kanlaus} concerning the maximal ideals of the Banach algebras $\mathscr{B}(X)$ for certain classical Banach spaces $X$. From this perspective, we complement results (\cite{gr, luft}) of Gramsch and Luft, who classified all closed ideals of operators acting on Hilbert spaces, and a result (\cite{daws}) of Daws, who classified all closed ideals of operators acting on $c_0(\lambda)$ and $\ell_p(\lambda)$ for $\lambda$ uncountable and $p\in [1,\infty)$. Nevertheless, bearing in mind that the similar problem of classifying all closed ideals of operators acting on $\ell_\infty(\lambda)$ is most likely intractable (as this would require, in particular, understanding \emph{all} injective Banach spaces), we offer instead a complete description of the maximal ideals for the space $\ell_\infty(\lambda)$ and its closed subspace
\[ \ell_\infty^c(\lambda) = \{z\in \ell_\infty(\lambda)\colon z(\alpha)\neq 0{\rm{}\, for\; at\; most\; countably\; many\,}\alpha<\lambda\}.\]
(For consistency of notation we sometimes denote $\ell_\infty$ by $\ell_\infty^c(\omega)$.) Our first result then reads as follows.
\begin{theorem}\label{uniquemax}Let $\lambda$ be an infinite cardinal number. Suppose that $X$ is either $\ell_\infty(\lambda)$ or $\ell_\infty^c(\lambda)$. Then $$\begin{array}{lcl}\mathscr{S}_X(X) &=& \{T\in \mathscr{B}(X)\colon T\text{ is not bounded below on any copy of }X\}\\
& = & \{T\in \mathscr{B}(X)\colon I_X\neq ATB\text{ for all }A,B\in \mathscr{B}(X)\}\end{array}$$
is the unique maximal ideal of $\mathscr{B}(X)$. \end{theorem}
(We explain the required notation and terminology in the next section, while proofs are postponed to the final section.)

Actually, we are able to describe all the closed ideals of $\ell_\infty^c(\lambda)$ containing the ideal of weakly compact operators---they consist precisely of operators that do not preserve copies of $\ell_\infty^c(\kappa)$ for some $\kappa$ not exceeding the cardinal successor of $\lambda$.

\begin{theorem}\label{listofideals2}Let $\lambda$ be an infinite cardinal number. Then every closed ideal of $\mathscr{B}(\ell_\infty^c(\lambda))$ that   contains the ideal of weakly compact operators is equal to $\mathscr{S}_{\ell_\infty^c(\kappa)}(\ell_\infty^c(\lambda))$ for some infinite cardinal number $\kappa\leqslant\lambda^+$. \end{theorem}

\begin{remark*}One cannot hope that a statement similar to Theorems~\ref{listofideals} or \ref{listofideals2} holds for the lattice of closed ideals of $\mathscr{B}(\ell_\infty(\lambda))$. Indeed, in general the lattice of closed ideals of $\mathscr{B}(\ell_\infty(\lambda))$ need not be linearly ordered. Let $E_1 = L_\infty(\{0,1\}^{\omega_2})$ and $E_2=\ell_\infty(\omega_1)$. Then both $E_1$ and $E_2$ embed into $\ell_\infty(\omega_2)$, but   $E_1$ does not embed into $E_2$ and  $E_2$ does not embed into $E_1$ (\cite[Theorem 1.7]{ros2}). Consider the (not necessarily closed) ideals $\mathscr{J}_i$ consisting of operators that admit a factorisation through $E_i$ ($i=1,2$). These sets are obviously closed under compositions from left and right, and, since $E_i\cong E_i \oplus E_i$ ($i=1,2$), they are also closed under addition, so they are indeed ideals of $\mathscr{B}(\ell_\infty(\omega_2))$. We \emph{claim} that the closures of $\mathscr{J}_i$ ($i=1,2$) are incomparable. Indeed, let $P_i$ be a projection on $\ell_\infty(\omega_2)$ with range isomorphic to $E_i$ ($i=1,2$), and suppose that $P_1\in \overline{\mathscr{J}_2}$. Since $P_1^2 = P_1$, we actually have $P_1\in \mathscr{J}_2$ (see \cite[Lemma 2.7]{laschza}). This, however, is impossible because it would imply that the range of $P_1$ is isomorphic to a subspace of $E_2$. The other case is symmetric.

This observation shows even more: there are 1-injective Banach spaces $E$ of the form $L_\infty(\mu)$ for which $\mathscr{B}(E)$ has more than one maximal ideal. It is readily seen that $E=E_1\oplus E_2$ is such an example. Here $E=L_\infty(\mu)$, where $\mu$ is the direct product measure of the Haar measure on $\{0,1\}^{\omega_2}$ with the counting measure on $\omega_1$.\end{remark*}

A key result in the proof of Theorem~\ref{uniquemax} is the following theorem, which, we believe, is of interest in itself. In order to state it, we require a piece of terminology.

A Banach space $X$ is \emph{complementably homogeneous} if for each closed subspace $Y$ of $X$ such that $Y$ is isomorphic to $X$, there exists a closed, complemented subspace $Z$ of $X$ such that $Z$ is contained in $Y$ and $Z$ is isomorphic to $X$. That $\ell_1(\lambda)$ is complementably homogeneous follows directly from Lemma~\ref{rosl1}(ii). The Banach spaces $c_0(\lambda)$ have actually a stronger property: every isomorphic copy of $c_0(\lambda)$ in $c_0(\lambda)$ is complemented (\cite[Proposition 2.8]{argyrosetal}). For $p\in (1,\infty)$ complementable homogeneity of $\ell_p(\lambda)$ is also well-known and easy to prove; we shall include a proof of that fact (Proposition~\ref{lpch}) for the sake of completeness.

\begin{theorem}\label{complemented}For each infinite cardinal number $\lambda$, the Banach space $\ell_\infty^c(\lambda)$ is complementably homogeneous.\end{theorem}

We then employ Theorem~\ref{complemented} to characterise, up to isomorphism, all infinite-dimensional complemented subspaces of $\ell_\infty^c(\lambda)$, where $\lambda$ is any uncountable cardinal number.

\begin{theorem}\label{maincomplemented}Let $\lambda$ be an infinite cardinal number. Then every infinite dimensional, complemented subspace of $\ell_\infty^c(\lambda)$ is isomorphic either to $\ell_\infty$ or to $\ell_\infty^c(\kappa)$ for some cardinal $\kappa\leqslant \lambda$. Consequently, $\ell_\infty^c(\lambda)$ is a primary Banach space.
\end{theorem}

Let us remark that $\ell_\infty^c(\lambda)$ is a hyperplane in the space $E:=\mbox{span}\, \{\mathbf{1}_\lambda, \ell_\infty^c(\lambda)\}$. The space $E$ in turn is easily seen to be isomorphic to its hyperplanes as it contains $\ell_\infty$. More importantly, $E$ is a unital sub-C*-algebra of $\ell_\infty(\lambda)$ (the operations are pointwise), hence by the Gelfand--Naimark duality, $E\cong C(K)$ for some compact Hausdorff space $K$. The space $E$ has one more incarnation---it is isometrically isomorphic to the space $L_\infty(\mu)$, where $\mu$ is the counting measure on $\lambda$ restricted to the $\sigma$-field of sets that are countable or have countable complement. The space $L_\infty(\mu)$ is not isomorphic to a complemented subspace of a dual space. Indeed, if it were isomorphic to a complemented subspace of a dual space, then it would be complemented in its second dual, which is an injective space, so $L_\infty(\mu)$ would be injective. However, it is an observation of Pe\l czy\'{n}ski and Sudakov (\cite{pelsu}) that $\ell_\infty^c(\lambda)$ ($\lambda$ uncountable) is not complemented in $\ell_\infty(\lambda)$, so it is not injective.\medskip

It is noteworthy that there are very few kinds of $C(K)$-spaces with completely understood complemented subspaces. To the best of our knowledge, these are:
\begin{itemize}
\item $c_0(\lambda)$ for any set $\lambda$; these spaces are isomorphic to $C(\lambda^1)$ where $\lambda^1$ is the one-point compactification of the discrete set $\lambda$. Every infinite-dimensional, complemented subspace of $c_0(\lambda)$ is isomorphic to $c_0(\kappa)$ for some cardinal number $\kappa\leqslant \lambda$ (see \cite{granero} or~\cite[Proposition~2.8]{argyrosetal}).
\item $C[0,\omega^\omega]$, where $[0,\omega^\omega]$ carries the order topology. Every infinite-dimensional, complemented subspace of $C[0,\omega^\omega]$ is isomorphic either to $c_0$ or to $C[0,\omega^\omega]$ (\cite[Theorem 3]{benyamini}).

No further separable examples---apart from $c_0$ and $C[0,\omega^\omega]$---of infinite dimensional $C(K)$-spaces with completely understood complemented subspaces are known.
\item $C(\beta \mathbb{N}) \cong \ell_\infty$. Every infinite-dimensional, complemented subspace of $\ell_\infty$ is isomorphic to $\ell_\infty$; this is Lindenstrauss' theorem (\cite{lin}).
\item $C(K)$, where $K$ is the compact scattered space constructed by Koszmider under the Continuum Hypothesis (\cite{Ko}); it is the Stone space of a Boolean algebra generated by a certain almost disjoint family of subsets of $\mathbb{N}$. Every infinite-dimensional, complemented subspace of $C(K)$ is isomorphic either to $c_0$ or $C(K)$.
\item Spaces of the form $C(K\sqcup L)\cong C(K)\oplus C(L)$, where $K$ is scattered, $C(L)$ is a Grothendieck space (equivalently, every operator from $C(L)$ to a separable space is weakly compact; for definition and basic properties of Grothendieck spaces see, \emph{e.g.}, \cite[p. 150]{diestel}), and all the complemented subspaces of $C(K)$ and $C(L)$ are classified. (So far, for $K$ we can take the one-point compactification of a discrete set, $[0,\omega^\omega]$, or the above-mentioned space constructed by Koszmider; and for $L $ we can take $ \beta \mathbb{N}$, so that  $C(L)$ is isometrically  isomorphic to $\ell_\infty$, a prototypical example of a Grothendieck space; see \cite{groth}.) Each complemented subspace of $C(K)\oplus C(L)$ is isomorphic to a space of the form $X\oplus Y$, where $X$ is a complemented subspace of $C(K)$ and $Y$ is complemented in $C(L)$.

Indeed, if $K$ is scattered, then $C(K)^*\cong \ell_1(|K|)$, so it follows from \cite[Corollary 2]{hj} that the unit ball of $C(K)^*$ is weak*-sequentially compact. Applying \cite[Corollary 5 on p.~150]{diestel}, we infer that every operator $T\colon C(L)\to C(K)$ is weakly compact and hence strictly singular by \cite{pelczynski}. It follows then from \cite[Theorem 1.1]{wojtaszczyk} that every complemented subspace of $C(K)\oplus C(L)$ is of the desired form.
\end{itemize}

We therefore extend the above list by adding the following classes of examples:
\begin{itemize}
\item $\ell_\infty^c(\lambda)$, for any uncountable cardinal number $\lambda$,
\item $\ell_\infty^c(\lambda)\oplus C(K)$, where $C(K)$ is isomorphic to one of the spaces $c_0(\kappa)$, $C[0,\omega^\omega]$ or the above-mentioned space constructed by Koszmider. Indeed, for any cardinal number $\lambda$, $\ell_\infty^c(\lambda)$ is a Grothendieck space (see Proposition~\ref{grothendieck} below), so we can use the antepenultimate clause to deduce the claim.\end{itemize}

As a by-product of our investigations, we give an alternative approach  to Daws' description of all closed ideals of $\mathscr{B}(c_0(\lambda))$ and $\mathscr{B}(\ell_p(\lambda))$ for $\lambda$ uncountable and $p\in [1,\infty)$.
\begin{theorem}\label{listofideals}Let $\lambda$ be an infinite cardinal number and let $p\in [1,\infty)$. Then every non-zero, proper closed ideal of $\mathscr{B}(c_0(\lambda))$ and $\mathscr{B}(\ell_p(\lambda))$ is of the form $\mathscr{S}_{c_0(\kappa)}(c_0(\lambda))$ and $\mathscr{S}_{\ell_p(\kappa)}(\ell_p(\lambda))$, respectively, for some infinite cardinal number $\kappa\leqslant\lambda$. \end{theorem}

We thank the referee for an unusually detailed and helpful report.

\section{Preliminaries}In this paper, we consider Banach spaces over the fixed scalar field either of real or of complex numbers. Our terminology is standard and follows mainly \cite{ak} and \cite{lt}. By  \emph{operator}, we mean a bounded, linear operator acting between Banach spaces. Let $X$ and $Y$ be Banach spaces. We denote by $\mathscr{B}(X,Y)$ the Banach space of all operators from $X$ to $Y$ and write $\mathscr{B}(X)$ for $\mathscr{B}(X,X)$. An operator $P\in \mathscr{B}(X)$ is a \emph{projection} if $P^2 = P$.

Let $T\colon X\to Y$ be an operator. We say that $T$ is \emph{bounded below} if there is a constant $c>0$ such that $\|Tx\|\geqslant c\|x\|$ for all $x\in X$. It is clear an operator is bounded below if and only if it is injective and has closed range, in which case it is an isomorphism onto its range.

Let $E$ be a Banach space. We say that on operator $T$ acting between Banach spaces is \emph{bounded below on a copy of }$E$ if there exists a subspace $E_0$ of the domain of $T$ such that $E_0$ is isomorphic to $E$ and the operator $T|_{E_0}$ is bounded below. For each pair of Banach spaces $X,Y$ we set
\[ \mathscr{S}_E(X,Y) = \{ T\in \mathscr{B}(X,Y)\colon T\text{ is not bounded below on any copy of }E\},\]
and write $\mathscr{S}_E(X) = \mathscr{S}_E(X,X)$. We call operators which belong to $\mathscr{S}_E(X,Y)$ $E$-\emph{singular operators}. We denote by $\mathscr{S}_E$ the class of all $E$-singular operators acting between arbitrary Banach spaces. In general, the set $\mathscr{S}_E(X)$ need not be closed under addition. This is readily seen in the case where $E=X = \ell_p\oplus \ell_q$ for $1\leqslant p<q<\infty$.

We use von Neumann's definition of an ordinal number. We identify cardinal numbers with initial ordinal numbers. For instance, $\omega_1$ is the first uncountable ordinal, $\omega_2$ is the second one and so on. If $\lambda$ is a cardinal and $\alpha<\lambda$ we often write $[0,\alpha)$ and $(\alpha, \lambda)$ for the set of ordinals less than $\lambda$ which are, respectively, less than $\alpha$ and greater than $\alpha$. Even though $\alpha$ and $[0,\alpha)$ are equal, we distinguish these symbols to avoid confusion in the situation when we think of $\alpha$ as an element of $\lambda$ and in the situation when it is considered a subset of it. The \emph{cofinality} of a cardinal number $\lambda$ is the minimal cardinality $\kappa$ of a family of cardinals strictly less than $\lambda$, say $\{\lambda_\alpha\colon \alpha<\kappa\}$, such that $\lambda = \sum_{\alpha<\kappa} \lambda_\alpha$. The cofinality of $\omega_1$ is $\omega_1$ while the cofinality of $\omega_{\omega}$ is countable. If $\Lambda$ is a set, we denote by $|\Lambda|$ the cardinality of $\Lambda$. For a cardinal number $\lambda$, we denote by $\lambda^+$ its cardinal successor; for instance $\omega_1^+=\omega_2$ etc.

For a subset $Y$ of a topological space $X$ we define the \emph{density} of $Y$, $\mbox{dens}\, Y$,  to be the minimal cardinality of a dense subset of $Y$. Let $X$ be a Banach space and let $Y\subseteq X$ be a closed subspace. 

Let $\lambda$ be an infinite cardinal and let $E_\lambda$ denote one of the spaces: $c_0(\lambda)$, $\ell_p(\lambda)$ ($p\in [1,\infty]$) or $\ell_\infty^c(\lambda)$. For each subset $\Lambda\subseteq \lambda$ the following formula defines a contractive projection $P_\Lambda$ on $E_\lambda$:

\begin{equation}\label{projdef}(P_\Lambda f)(\alpha) = \left\{\begin{array}{ll}f(\alpha), &\text{if }\alpha\in \Lambda\\ 0, &\text{if }\alpha\notin \Lambda\end{array} \right.\qquad (f\in E_\lambda).\end{equation}
We then denote the range of $P_\Lambda$ by $E_\Lambda$. Of course, the range of $P_\Lambda$ is isometrically isomorphic to $E_{|\Lambda|}$. We denote by $R_\Lambda\colon E_\lambda\to E_{\Lambda}$ the \emph{restriction operator}, which identifies an element $f$ in the range of $P_\Lambda$ with the restriction of $f$, $R_\Lambda f$, to the set $\Lambda$. For $f\in E_\lambda$ we define the \emph{support} of $f$ by $\mbox{supp}\, f = \{\alpha<\lambda\colon f(\alpha)\neq 0\}$. A function $f$ is \emph{supported} in a set $\Lambda$ if $\mbox{supp}\, f\subseteq \Lambda$.

We shall also require the following lemma due to Rosenthal (\cite[Proposition 1.2, Corollary on p.~29 and Remark 1 on p.~30]{ros1}).

\begin{lemma}\label{rosl1}Let $\lambda$ be an infinite cardinal number and let $X$ be a Banach space.
\begin{enumerate}
\item[(i)] Suppose that $T\colon \ell_\infty(\lambda)\to X$ is an operator such that $T|_{c_0(\lambda)}$ is bounded below. Then there exists a set $\Lambda\subset \lambda$ of cardinality $\lambda$ such that $T|_{\ell_\infty(\Lambda)}$ is bounded below. Consequently, if $Z$ is an injective Banach space and $T\colon Z\to X$ is an operator that is bounded below on a copy of $c_0(\lambda)$, then $T$ is bounded below on a copy of $\ell_\infty(\lambda)$.
\item[(ii)] Suppose that $A\subseteq X$ is a closed subspace and $T\colon X\to \ell_1(\lambda)$ is an operator such that for some $
\delta> 0$ and some bounded set $\{y_\alpha\colon \alpha<\lambda\}\subset A$ we have
\[ \delta\leqslant \|Ty_\alpha - Ty_\beta\|\qquad (\alpha,\beta < \lambda, \alpha\neq \beta).\]
Then $T$ is bounded below on some subspace of $A$ that is isomorphic to $\ell_1
(\lambda)$ and complemented in $X$.
\item[(iii)] Suppose that $T\colon c_0(\lambda)\to X$ is an operator such that $$\inf\{\|Te_\alpha\|\colon \alpha<\lambda\}>0,$$ where $\{e_\alpha\colon \alpha<\lambda\}$ is the canonical basis of $c_0(\lambda)$. Then there exists a set $\Lambda\subset \lambda$ of cardinality $\lambda$ such that $T|_{c_0(\Lambda)}$ is bounded below.
\end{enumerate}
\end{lemma}

We say that a Banach space is \emph{primary} if whenever $X$ is isomorphic to $Y\oplus Z$, then at least one of the spaces $Y$ or $Z$ is isomorphic to $X$.

\section{Auxiliary results}\label{section2}

\begin{proposition}\label{proposition2NEW}
Let $\lambda$ and $\tau$ be infinite cardinal numbers and suppose that $\lambda$ is uncountable. Let $T\colon \ell_\infty^c(\tau)\to Y$ be an operator that is not bounded below on any sublattice isometric to $c_0(\lambda)$. Then for every $\varepsilon > 0$ there is subset $\Lambda$ of $\tau$ so
that $|\Lambda| < \lambda$ and $$
\| T R_{\tau\setminus \Lambda} \| \leqslant \varepsilon.$$

Consequently, if $Y= \ell_\infty^c(\tau)$ and  $T$ is a projection onto a subspace $X$, then $X$ is isomorphic to a complemented subspace of $\ell_\infty^c(\kappa)$ for some cardinal number $\kappa<\lambda$.\end{proposition}

\begin{proof}
Let $(f_\a)_{\a \in \Gamma}$ be a collection of disjointly supported unit vectors in
$\ell_\infty^c(\tau)$ that is maximal with respect to the property that
$\|T f_\a \|\geqslant \varepsilon$ for each $\a$. Let $\Lambda$ be the union of the
supports of the $f_\a$ ($\alpha\in \Gamma$). In the case where $\Gamma$ is finite, $|\Lambda|<\lambda$ as $\Lambda$ is countable. In the case where $\Gamma$ is infinite, by the
hypothesis on $T$ and Lemma~\ref{rosl1}(iii), we have $|\Gamma| < \lambda$, so that also $|\Lambda| < \lambda$. By the maximality of
the collection $(f_\alpha)_{\alpha\in \Gamma}$, if $f$ is a unit vector whose support is contained
in $\lambda\setminus \Lambda$, then $\|T f\| < \varepsilon$, which implies that $\| T R_{\tau\setminus \Lambda} \| \leqslant \varepsilon.$

For the ``consequently" statement, suppose that $T$ is a projection onto a
subspace $X$.  Then
$$
I_X = (TR_{\tau\setminus \Lambda} + TR_\Lambda)|_{X},
$$
so
$$
\|I_X - (TR_\Lambda)|_X\| \leqslant \varepsilon.
$$
Hence if $\varepsilon < 1$, there is an operator $U$ on $X$ so that
$$
I_X = (UTR_\Lambda)|_{X}.
$$
Thus $I_X$ factors through $\ell_\infty^c(|\Lambda|)$.
  \end{proof}
  The next proposition, which is surely known, puts the hypothesis of Proposition \ref{proposition2NEW} into perspective.
 \begin{proposition}\label{perspective}
Let $\tau$ and $\lambda$ be cardinal numbers and suppose that $\lambda$ is uncountable. Let $E=\ell_\infty^c(\tau)$ and $E_\lambda=c_0(\lambda)$ or let $E=\ell_p(\tau)$ and $E_\lambda=\ell_p(\lambda)$ for some $1<p<\infty$. Suppose that $(x_\alpha)_{\alpha\in \lambda}$ is a transfinite sequence in $E$ that is equivalent to the unit vector basis of $E_\lambda$. Then there is a set $\Lambda \subset \lambda $ with $|\Lambda |= \lambda$ such that $(x_\alpha)_{\alpha\in \Lambda}$ have disjoint supports.\smallskip

Consequently, 
\begin{itemize}
\item[(i)] if $T\colon E\to E_\tau$ is a bounded, linear operator, then there is a set $\Lambda \subset \lambda $ with $|\Lambda |= \lambda$ such that $(Te_\alpha)_{\alpha\in \Lambda}$ have disjoint supports;
\item[(ii)] if $\tau < \lambda$, then $c_0(\lambda)$ is not isomorphic to a subspace of $ \ell_\infty^c(\tau)$.\end{itemize}
  \end{proposition}

  \begin{proof}
%   We treat the case $E=  \ell_\infty^c(\tau)$;  the other cases are similar. 
Take $\Lambda \subset \lambda$ maximal with respect to the property that the vectors $(x_\alpha)_{\alpha\in \Lambda}$ have disjoint supports.  Let $\Gamma$ be the union of the supports of  $(x_\alpha)_{\alpha\in \Lambda}$. 
Suppose, for contradiction, that  $|\Lambda|<|\lambda|$. Then also $|\Gamma|<|\lambda|$ since $\lambda $ is uncountable.  
  By the  maximality of $(x_\alpha)_{\alpha\in \Lambda}$, each $x_\alpha$ for $\alpha\in \lambda \setminus\Lambda$  is non-zero at some coordinate in $\Gamma$. Consequently, 
  uncountably many $x_\alpha$ are non-zero at the same coordinate $\gamma \in \Gamma$. So, without loss of generality, we have for some $\varepsilon >0$ a sequence of distinct $x_{\alpha_n}$ so that for all $n$ the real part of $x_{\alpha_n}(\gamma)$  is larger than $\varepsilon$.  Then for every $N$, 

\[\Big\|\sum_{n=1}^N x_{\alpha_n}\Big\|  \geqslant \left\{\begin{array}{ll}\varepsilon \cdot N, & \text{if}\, E =\ell_\infty^c(\tau), \\ \frac{\varepsilon \cdot N}{N^{1/p}}, & \text{if}\, E=\ell_p(\tau),\end{array}\right.\] contradicting the fact that $(x_\alpha)_{\alpha\in \lambda}$ is equivalent to the unit vector basis for $E_\lambda$.
  \end{proof}

We thus obtain an analogue of Lemma~\ref{rosl1}(iii) in the case of $\ell_p(\lambda)$ for $p\in (1,\infty)$.

\begin{corollary}\label{lemma21lp}Let $\lambda$ be an uncountable cardinal number and let $p\in (1,\infty)$. Suppose that $T\colon \ell_p(\lambda)\to \ell_p(\lambda)$ is a bounded, linear operator such that $\inf\{\|Te_\alpha\|\colon \alpha<\lambda\}>0$. Then there is a set $\Lambda \subset \lambda $ with $|\Lambda |= \lambda$ such that $TR_\Lambda$ is bounded below.\end{corollary}   
  
    Although not needed for the present paper, it is worth remarking that the ``consequently" statement in Proposition \ref{perspective} can be improved, as is shown by the following (probably known) simple proposition.
  
   \begin{proposition}\label{intoellinfty}
   Let $\tau < \lambda$ be cardinal numbers. Then $c_0(\lambda)$ is not isomorphic to a subspace of $\ell_\infty(\tau)$.
   \end{proposition}
   \begin{proof}
   Since the other cases are clear, we assume that $\tau$ is infinite and hence that $\lambda$ is uncountable. Suppose, for contradiction, that $(x_\alpha)_{\alpha\in \lambda}$ is a set of vectors in $\ell_\infty(\tau)$ that is equivalent to the unit vector basis of $c_0(\lambda)$. For each $t \in \tau$, the set $A_t$ of all $\alpha \in \lambda$ for which $x_\alpha(t)$ is not zero must be countable (this was already used in the proof of   Proposition \ref{perspective}). But then the cardinality of $\bigcup_{t\in \tau} A_t$ is at most 
   $\omega \cdot \tau = \tau <\lambda$. 
   \end{proof}
   
We also obtain a characterisation of complemented subspaces of $\ell_\infty^c(\omega_1)$ that do not embed into $\ell_\infty$.
 
\begin{corollary}\label{isomorphictoellinfty}Every infinite-dimensional, complemented subspace of $\ell_\infty^c(\omega_1)$ that does not contain a subspace isomorphic to $c_0(\omega_1)$ is isomorphic to $\ell_\infty$.\end{corollary}

\begin{proof}Let $X$ be an infinite-dimensional, complemented subspace of $\ell_\infty^c(\omega_1)$ that does not contain a subspace isomorphic to $c_0(\omega_1)$. Let $P$ be any projection onto $X$. By Proposition~\ref{proposition2NEW}, $X$ is isomorphic to a complemented subspace of $\ell_\infty^c(\omega)=\ell_\infty$ and hence is isomorphic to $\ell_\infty$.\end{proof}  

In view of Corollary~\ref{isomorphictoellinfty}, in order to complete the classification of the complemented subspaces of $\ell_\infty^c(\lambda)$, that is to prove Theorem~\ref{complemented}, we need to show that a complemented subspace $X$ of $\ell_\infty^c(\lambda)$ that contains an isomorphic copy of  $c_0(\lambda)$ must be isomorphic to $\ell_\infty^c(\lambda)$. We begin by showing that such an $X$ must contain a subspace isomorphic to  $\ell_\infty^c(\lambda)$.

 \begin{proposition}\label{proposition3}
Let $\kappa$ and $\lambda$ be uncountable cardinal numbers and let $T\colon \ell_\infty^c(\lambda)\to \ell_\infty^c(\lambda)$ be an operator. 
Assume that $T$ is bounded below on a subspace of $ \ell_\infty^c(\lambda)$ that is isomorphic to $c_0(\kappa)$.
Then there is a subspace $Y$ of $\ell_\infty^c(\lambda)$ that is isometric to $\ell_\infty^c(\kappa)$ such that $T|_{Y}$ is bounded below.  \end{proposition}

\begin{proof}
By Proposition \ref{perspective}  there is a transfinite sequence $(x_\b)_{\b<\kappa}$ of disjointly supported unit vectors in  $\ell_\infty^c(\lambda)$ so that the restriction of $T$ to $\overline{{\rm span}}\,  \{x_\b\colon \b<\kappa\}$ is bounded below. 
Define an isometry $S\colon \ell_\infty^c(\kappa)\to \ell_\infty^c(\lambda)$ by $$(Sz)(\gamma) = \left\{\begin{array}{ll}z(\beta)x_\beta(\gamma), & \gamma\in {\rm supp}\,x_\beta,\\
0, & {\rm otherwise}\end{array}\right.\qquad(z\in \ell_\infty^c(\kappa), \beta<\kappa).$$
The condition on $T$ says that the operator $TS\colon \ell_\infty^c(\kappa) \to \ell_\infty^c(\lambda)$ is bounded below on $c_0(\kappa)$.
Regard $\ell_\infty^c(\kappa) $ as the subspace $\ell_\infty^c([0,\kappa))$ of $\ell_\infty(\lambda)$. By injectivity of $\ell_\infty(\lambda)$, the operator $TS$ has a norm-preserving extension to an operator $U\colon \ell_\infty(\lambda)\to \ell_\infty(\lambda)$    \cite[Proposition 2.f.2]{lt}. The operator $U$ is an
isomorphism on $c_0(\kappa)$, hence by Lemma~\ref{rosl1}(i), there is a subset $\Gamma$ of $\kappa$ with cardinality $\kappa$ such that the restriction of $U$ to $\ell_\infty(\Gamma)$ is an isomorphism. But $U$ maps $\ell_\infty^c(\Gamma)$ into $\ell_\infty^c(\lambda)$, which is to say that $TS$ is an isomorphism on $\ell_\infty^c(\Gamma)$. \ that $Y:= S[\ell_\infty^c(\Gamma)]$ is isometric to $\ell_\infty^c(\kappa)$ and $T|_{Y}$ is bounded below.
\end{proof}
The following proposition is essentially known in the sense that it can be easily deduced from a string of standard results concerning Grothendieck spaces; however, we offer here a direct proof.

\begin{proposition}\label{grothendieck}For every cardinal number $\lambda$, $\ell_\infty^c(\lambda)$ is a Grothendieck space.\end{proposition}
\begin{proof} If $\ell_\infty^c(\lambda)$ is not a Grothendieck space there exists an operator $T\colon \ell_\infty^c(\lambda)\to c_0$ that is not weakly compact (see \cite[p.~150]{diestel}). Then by \cite{pelczynski} there is a subspace $Y$ of $\ell_\infty^c(\lambda)$ isomorphic to $c_0$ such that $T|_{Y}$ is bounded below. Then $Y$ is contained in  $E_\Lambda$ for some countable set $\Lambda\subset \lambda$. In particular, $T|_{E_\Lambda}$ can be regarded as operator from $\ell_\infty$ into $c_0$ that is bounded below on a copy of $c_0$. By Lemma~\ref{rosl1}(i), $T$ must be bounded below on a copy of $\ell_\infty$; a contradiction as $T$ maps into $c_0$.\end{proof}

We need a consequence of \cite[Lemma 1.1]{ros1}.

\begin{lemma}\label{Rosenthal}
Let $\Gamma$ be uncountable and let $T$ be an operator on $\ell_\infty(\Gamma)$ such that $T$ is the identity on $c_0(\Gamma)$.  Then for every $\varepsilon > 0$  there is $\Gamma' \subset \Gamma$ with $|\Gamma'| = |\Gamma|$ so that $$\|(R_{\Gamma'} T)|_{\ell_\infty(\Gamma')} - I_{\ell_\infty(\Gamma')} \|  < \varepsilon,$$ where $R_{\Gamma'}\colon \ell_\infty(\Gamma) \to \ell_\infty(\Gamma')$ is the restriction operator.
\end{lemma}
\begin{proof}
We identify $\ell_\infty(\Gamma)^*$ with the finitely additive, scalar-valued measures on the $\sigma$-algebra of all subsets of $\Gamma$.  For $\gamma \in \Gamma$, set $\mu_\gamma = T^* \delta_\gamma$, where $ \delta_\gamma$ is the point mass measure at $\gamma$.  By \cite[Lemma 1.1]{ros1} there is $\Gamma' \subset \Gamma$ with $|\Gamma'| = |\Gamma|$ so that for all $\gamma \in \Gamma'$ we have $|\mu_\gamma| (\Gamma' \setminus \{\gamma\}) < \varepsilon$.
\end{proof}

\begin{proof}   
[Proof of Theorem~\ref{complemented}]
We need to show that if $X$ is a subspace of $\ell_\infty^c(\lambda)$ isomorphic to $\ell_\infty^c(\lambda)$, then $X$ contains a subspace isomorphic to $\ell_\infty^c(\lambda)$ that is complemented in $\ell_\infty^c(\lambda)$.The case where $\lambda$ is countable follows from injectivity of $\ell_\infty$, so let us suppose that $\lambda$ is uncountable. Let $\lessdot$ denote the order on $X$ that is induced from its isomorphism with $\ell_\infty^c(\lambda)$. 
Applying Proposition \ref{perspective} twice in succession,  we get a transfinite sequence $\{f_\gamma\colon \gamma < \lambda\}$ of unit vectors in $X$ that are disjoint both relative to the order structure on $X$ given by $\lessdot$ and the order structure on $\ell_\infty^c(\lambda)$ given by its pointwise ordering $<$. Let $Z$ be the lattice closure in $(\ell_\infty^c(\lambda), <)$ of the linear span of $\{f_\gamma\colon \gamma < \lambda\}$ (so that $Z$ is isometrically isomorphic to $\ell_\infty^c(\lambda)$ and there is a norm-one projection $Q$ from  $\ell_\infty^c(\lambda)$ onto $Z$).  By replacing $X$ itself with a subspace, we might as well assume that in $(X,\lessdot)$ the $f_\gamma$ form the unit vector basis for $c_0(\lambda)$ and hence that in $(X,\lessdot)$ the space $X$ is the pointwise closure of the $f_\gamma$. Let $J\colon (Z, <) \to (X,\lessdot)$ denote the natural surjective order isomorphism that is the identity on the set $\{f_\gamma\colon \gamma<\lambda\}$. Now $$(JQ)|_{X}\colon X \to X$$ is an operator on $X$ that is the identity on $\{f_\gamma\colon \gamma < \lambda\}$, so by Lemma \ref{Rosenthal} (applied to any extension of the operator to an operator on the $\ell_\infty(\lambda)$ lattice generated by $(X,\lessdot)$) we have a subset $\Lambda$ of $\lambda$ of cardinality $\lambda$ so that
\begin{equation}
\|(R_\Lambda JQ)|_{X_\Lambda} - I_{X_\Lambda}\| < \varepsilon
\end{equation}
where $X_\Lambda$ is the pointwise closure of $\{f_\gamma\colon \gamma \in \Lambda \}$ in $(X, \lessdot)$,  $DR_\Lambda$ is the restriction mapping on $(X,\lessdot)$, and $\varepsilon > 0$ can be as small as we want; in particular, $\varepsilon < 1$.  So there is an automorphism $U$ on $X_\Lambda$ so that $UR_\Lambda JQ$ is the identity on $X_\Lambda$ and hence $UR_\Lambda JQ$ is a projection from $\ell_\infty^c(\lambda)$ onto $X_\Lambda$. Since $X_\Lambda$ is isomorphic to $\ell_\infty^c(\lambda)$, this completes the proof.\end{proof}

The above proof of Theorem~\ref{complemented} can be adjusted to get the following stronger result.

\begin{proposition}\label{stronger}
Let $\kappa$ and $\lambda$ be infinite cardinal numbers. If $X$ is a subspace of $\ell_\infty^c(\lambda)$ that is isomorphic to $\ell_\infty^c(\kappa)$, then there is a subspace $Y$ of X that is isomorphic to $\ell_\infty^c(\kappa)$ and complemented in $\ell_\infty^c(\lambda)$.  
\end{proposition}

\begin{proposition}\label{lpch}
For each uncountable cardinal number $\lambda$ and $p\in [1,\infty)$ the Banach space $\ell_p(\lambda)$ is complementably homogeneous.
\end{proposition}
\begin{proof}The case $p=1$ is covered by Lemma~\ref{rosl1}(ii). When  $p\in (1,\infty)$, the result follows from Proposition \ref{perspective}.
\end{proof}

\begin{remark}\label{lpstronger}It is perhaps worthwhile to note that an analogue of Proposition~\ref{stronger} holds also for $E_\lambda=\ell_p(\lambda)$ ($p\in [1,\infty)$); that is, if $Y$ is a subspace of $E_\lambda$ isomorphic to $E_\kappa$ for some cardinal number $\kappa$, then there is a subspace $X\subseteq Y$ isomorphic to $E_\kappa$ that is complemented in $E_\lambda$. This is because every copy of $E_\kappa$ in $E_\lambda$ is contained in $E_\Lambda$ for some set $\Lambda\subset \lambda$ with $|\Lambda|=\kappa$.\end{remark}

\begin{proposition}\label{singular}Let $\lambda$ and $\kappa$ be infinite cardinal numbers, $p\in [1,\infty)$ and let $X$ be a Banach space. Then the sets
\[\mathscr{S}_{c_0(\lambda)}\big(X, \ell_\infty^c(\kappa)\big) \text{ and } \mathscr{S}_{\ell_p(\lambda)}\big(X, \ell_p(\kappa)\big) \]
are closed under addition.
\end{proposition}

\begin{proof}The case $\lambda=\omega$ follows from \cite[Proposition 2.5]{kanlaus}. Let us consider the case where $\lambda$ is uncountable.

Suppose that $T,S$ are in $\mathscr{S}_{c_0(\lambda)}\big(X, \ell_\infty^c(\kappa)\big)$ and assume, in search of contradiction, that $T+S$ is not in $\mathscr{S}_{c_0(\lambda)}\big(X, \ell_\infty^c(\kappa)\big)$. Then there exists a subspace $X_1\subseteq X$ isomorphic to $c_0(\lambda)$ and $\varepsilon > 0$ such that $\varepsilon\|z\|\leqslant \|(T+S)(z)\|$ for all $z\in X_1$. Arguing similarly as in the proof of Proposition~\ref{proposition2NEW}, we infer that there exists a subspace $Z_1\subset X_1$ isomorphic to $c_0(\lambda)$ such that $\|T|_{Z_1}\|\leqslant \tfrac{\varepsilon}{3}$. 
%(Formally, we apply Proposition~\ref{proposition2NEW} to $T|_{X_1}$ acting on $X_1$ renormed to $c_0(\lambda)$.) 
Repeating the above argument, we conclude that there is a subspace $Z_2\subseteq Z_1$ isomorphic to $c_0(\lambda)$ such that $\|S|_{Z_2}\|\leqslant \tfrac{\varepsilon}{3}$. Consequently, for all $z\in Z_2$ we have
\[\varepsilon\|z\|\leqslant \|Tz + Sz\| \leqslant \frac{\varepsilon}{3}\|z\|+\frac{\varepsilon}{3}\|z\| = \frac{2}{3}\varepsilon\|z\|; \]
a contradiction.\smallskip

The proof in the case of $\ell_p(\lambda)$ for $p\in (1,\infty)$ is analogous (here we use Corollary~\ref{lemma21lp} instead of Lemma~\ref{rosl1}(iii) when arguing as in the proof of Proposition~\ref{proposition2NEW}).\smallskip

The proof of the case of $p=1$ is actually contained in \cite{ros1}. Indeed, suppose contrapositively that $T,S\colon X\to \ell_1(\kappa)$ are operators such that $T+S$ is not in $\mathscr{S}_{\ell_1(\lambda)}\big(X, \ell_1(\kappa)\big)$. 
Let $Y$ be a copy of $\ell_1(\lambda)$ in $X$ such that $(T+S)_{|Y}$ is an isomorphism. By \cite[Theorem 3.3]{ros1}, there is a set of unit vectors $\{x_\alpha\colon \alpha<\lambda\}$ in $Y$ and disjoint sets $E_\alpha\subset \kappa$ ($\alpha<\lambda$) such that for some $\delta>0$ $\|((T+S)x_\alpha)|_{E_\alpha}\|\geqslant \delta$. It follows that for some subset $\Lambda\subseteq \lambda$ of cardinality $\lambda$ and for either $T$ or $S$, say $T$, $\|(Tx_\alpha)|_{E_\alpha}\|\geqslant \delta/2$ for all $\alpha\in \Lambda$. It now follows from \cite[Propositions 3.2 and 3.1]{ros1} that $\Lambda$ contains a subset $\Lambda^\prime$ of the same cardinality such that $T$ is bounded below on the closed linear span of $\{x_\alpha\colon \alpha\in \Lambda^\prime\}$.
\end{proof}

\begin{remark*}
One can improve a bit the $\ell_1$-version of the previous proposition: If $X$ contains an isomorphic copy of $\ell_1(\lambda)$ then for each $\varepsilon>0$ it contains a $(1+\varepsilon)$-isomorphic copy of $\ell_1(\lambda)$ that is complemented by means of a projection of norm at most $1+\varepsilon$.

This follows from a simple combination of the non-separable counter-part of the James distortion theorem for $\ell_1$ (\cite{james}, \emph{cf.} \cite{giesy}), which provides the $(1+\varepsilon)$-isomorphic copy of $\ell_1(\lambda)$ together with Dor's result \cite[Theorem A]{dor}, which provides the needed projection.

This remark is provided here because it may be useful in questions related to the classification of commutators. Similar facts were useful when classifying the commutators in, say, the algebra of bounded operators on $\ell_1$.
\end{remark*}

Proposition~\ref{singular} can be strengthened in the case of $c_0(\lambda)$:

\begin{proposition}\label{operatorideal}For each infinite cardinal number $\lambda$, $\mathscr{S}_{c_0(\lambda)}$ forms a closed operator ideal. \end{proposition}
\begin{proof}Let $X$ and $Y$ be Banach spaces. The set $\mathscr{S}_{c_0(\lambda)}(X,Y)$ is closed in $\mathscr{B}(X,Y)$ and obviously is closed under taking compositions with other operators (whenever these make sense). The only non-trivial thing is to verify that it is indeed closed under addition.

Let $T,S\in \mathscr{B}(X,Y)$ be operators such that $T+S$ is bounded below by $\delta>0$ on some subspace $X_0\subseteq X$ isomorphic to $c_0(\lambda)$. Let $\{e_\alpha\colon \alpha<\lambda\}$ be a transfinite sequence in $X_0$ equivalent to the canonical basis of $c_0(\lambda)$. Since
\[\delta \leqslant \|Te_\alpha + Se_\alpha\|\]
for all $\alpha<\lambda$, there is a set $\Lambda$ of cardinality $\lambda$ such that $\|Te_\alpha\|\geqslant \frac{\delta}{2}$ or $\|Se_\alpha\|\geqslant \frac{\delta}{2}$ for all $\alpha\in \Lambda$. It follows then from Lemma~\ref{rosl1}(iii) that at least one of the operators $T$ or $S$ does not belong to $\mathscr{S}_{c_0(\lambda)}(X,Y)$.\end{proof}

\begin{theorem}\label{trueideals}Let $\kappa$ and $\lambda$ be infinite cardinals and let $p\in [1,\infty]$. Then $\mathscr{A}$ is a closed ideal of $\mathscr{B}$ if
\begin{itemize}
\item[] $\mathscr{A} = \mathscr{S}_{c_0(\kappa)}(c_0(\lambda))$ and $\mathscr{B} = \mathscr{B}(c_0(\lambda))$;
\item[] $\mathscr{A} = \mathscr{S}_{\ell_p(\kappa)}(\ell_p(\lambda))$ and $\mathscr{B} = \mathscr{B}(\ell_p(\lambda))$;
\item[] $\mathscr{A} = \mathscr{S}_{\ell_\infty^c(\kappa)}(\ell_\infty^c(\lambda))$ and $\mathscr{B} = \mathscr{B}(\ell_\infty^c(\lambda))$.
\end{itemize}
Moreover,
\begin{equation}\label{moreover}\mathscr{S}_{\ell_\infty^c(\kappa)}(\ell_\infty^c(\lambda)) =  \mathscr{S}_{c_0(\kappa)}(\ell_\infty^c(\lambda)).\end{equation}\end{theorem}
\begin{proof}The first clause as well as the second one in the case of $p\in [1,\infty)$ follow directly from Proposition~\ref{singular}.

We claim that $\mathscr{S}_{\ell_\infty^c(\kappa)}(\ell_\infty^c(\lambda)) =  \mathscr{S}_{c_0(\kappa)}(\ell_\infty^c(\lambda))$, which is plainly enough as Proposition~\ref{singular} asserts that $\mathscr{S}_{c_0(\kappa)}(\ell_\infty^c(\lambda))$ is a closed ideal of $\mathscr{B}(\ell_\infty^c(\lambda))$. Let us suppose that $T\notin  \mathscr{S}_{c_0(\kappa)}(\ell_\infty^c(\lambda))$. If $\kappa$ is countable, then arguing as in the proof of Proposition~\ref{grothendieck}, we infer that $T$ preserves a copy of $\ell_\infty$. In the case where $\kappa$ is uncountable, Proposition~\ref{perspective} gives us a transfinite sequence of disjointly supported unit vectors $(x_\beta)_{\beta<\kappa}$ such that $T$ is bounded below on $\overline{\rm span}\{x_\beta\colon \beta<\kappa\}$. By Proposition~\ref{proposition3}, $T$ is bounded below on some copy of $\ell_\infty^c(\kappa)$, hence $T\notin \mathscr{S}_{\ell_\infty^c(\kappa)}(\ell_\infty^c(\lambda))$. On the other hand, $\mathscr{S}_{\ell_\infty^c(\kappa)}(\ell_\infty^c(\lambda))\supseteq  \mathscr{S}_{c_0(\kappa)}(\ell_\infty^c(\lambda))$ trivially, so the proof of the claim is complete.

Let us consider now the remaining case of $\mathscr{A} = \mathscr{S}_{\ell_\infty(\kappa)}(\ell_\infty(\lambda))$ and $\mathscr{B} = \mathscr{B}(\ell_\infty(\lambda))$. Suppose contrapositively that $T,S$ are operators on $\ell_\infty(\lambda)$ such that $T+S\notin \mathscr{S}_{\ell_\infty(\kappa)}(\ell_\infty(\lambda))$. Let $X\subseteq \ell_\infty(\lambda)$ be a copy of $c_0(\kappa)$ on which $T+S$ is bounded below. Since $\mathscr{S}_{c_0(\lambda)}$ is an operator ideal (Proposition~\ref{operatorideal}), either $T$ or $S$ is bounded below on some copy of $c_0(\lambda)$. By Lemma~\ref{rosl1}(i), at least one of those operators is bounded below on a copy of $\ell_\infty(\lambda)$, which means that either $T$ or $S$ is not in $\mathscr{S}_{\ell_\infty(\kappa)}(\ell_\infty(\lambda))$.\end{proof}

The next lemma is a counter-part of \cite[Proposition 5.1]{daws}.
\begin{lemma}\label{lemmadaws}Let $\kappa$ and $\lambda$ be infinite cardinal numbers and suppose that $\kappa$ is uncountable and is not a successor of any cardinal number. If $E_\lambda$ is one of the spaces $c_0(\lambda)$, $\ell_\infty^c(\lambda)$ or $\ell_p(\lambda)$ for some $p\in [1,\infty)$, then
$$\overline{\bigcup_{\rho<\kappa} \mathscr{S}_{E_\rho}(E_\lambda)} = \mathscr{S}_{E_\kappa}(E_\lambda).$$
Consequently, if the cofinality of $\kappa$ is uncountable $$\bigcup_{\rho<\kappa} \mathscr{S}_{E_\rho}(E_\lambda) = \mathscr{S}_{E_\kappa}(E_\lambda).$$ \end{lemma}
\begin{proof}Only the inclusion from right to left requires justification. Certainly we may suppose that $\kappa$ is uncountable as otherwise the left hand side is trivially equal to the right hand side.\smallskip

Take $T\in  \mathscr{S}_{E_\kappa}(E_\lambda)$. We claim that for each $\varepsilon>0$ there is a set $\Lambda$ with $|\Lambda| < \kappa$ such that $\|P_\Lambda T - T\| < \varepsilon$. We split the proof of this claim into two subcases.\smallskip
\begin{itemize}
\item Let us consider first the case where $E_\lambda= \ell_1(\lambda)$. Assume that the assertion fails for a certain $\varepsilon > 0$. Then there is a set of unit vectors $\{x_\alpha\colon \alpha<\kappa\}$ in $Y$ and finite disjoint sets $E_\alpha\subset \lambda$ ($\alpha<\kappa$) such that for some $\delta\in (0, \varepsilon)$ we have $\|Tx_\alpha|_{E_\alpha}\|\geqslant \delta$ ($\alpha<\kappa$). It now follows from \cite[Propositions 3.2 and 3.1]{ros1} that $\kappa$ contains a subset of the same cardinality, ${\rm K}$ say, such that $T$ is bounded below on $\overline{{\rm span}}\, \{x_\alpha\colon \alpha\in {\rm K}\}\cong \ell_1(\kappa)$.\smallskip

\item Assume that our assertion is not true in the remaining cases, which means that for some $\varepsilon_0>0$ and all sets $\Lambda$ with cardinality less than $\kappa$ we have $$\|P_{\lambda \setminus \Lambda}T\| = \|P_\Lambda T - T\| \geqslant \varepsilon_0.$$ If $E_\lambda = c_0(\lambda)$ or $E_\lambda = \ell_p(\lambda)$ for some $p\in (1,\infty)$, this is a contradiction. Indeed, otherwise by Proposition~\ref{perspective}(i), we would have a family of disjointly supported unit vectors $\{z_\gamma\colon \gamma<\kappa\}$ such that $\{Tz_\gamma\colon \gamma<\kappa\}$ are also disjointly supported and have norm at least $\varepsilon_0/2$. In particular, $T$ would be bounded below on a copy of $E_\kappa$ against the assumption. For $E_\lambda = \ell_\infty^c(\lambda)$, as in the case of $c_0(\lambda)$, there would exist a family of disjointly supported unit vectors $\{z_\gamma\colon \gamma<\kappa\}$ such that $\{Tz_\gamma\colon \gamma<\kappa\}$ are also disjointly supported and have norm at least $\varepsilon_0/2$. Thus $T$ would preserve a copy of $c_0(\lambda)$, hence by \eqref{moreover}, also a copy of $\ell_\infty^c(\kappa)$. \end{itemize}\smallskip

Since $P_{\Lambda} T\in \mathscr{S}_{E_\rho}(E_\lambda)$ ($n\in \mathbb{N}$), we conclude that $T$ belongs to the closure of the ideal $\bigcup_{\rho<\kappa} \mathscr{S}_{E_\rho}(E_\lambda)$ as desired. \smallskip

The final assertion follows from the well-known fact that in first-countable spaces increasing unions of  closed sets that are well-ordered by inclusion are closed when $\kappa $ has uncountable cofinality.\end{proof}

\begin{lemma}\label{lemmaW}For each infinite cardinal number $\lambda$ we have $$\mathscr{W}(\ell_\infty^c(\lambda)) = \mathscr{S}_{\ell_\infty}(\ell_\infty^c(\lambda)),$$ where $\mathscr{W}$ denotes the ideal of weakly compact operators.  \end{lemma}
\begin{proof}Let $T$ be an operator on $\ell_\infty^c(\lambda)$. Then by \cite{pelczynski}, $T$ is not weakly compact if and only if $T$ is bounded below on some copy of $c_0$. Suppose that $T$ is not weakly compact and let $F$ be a copy of $c_0$ that is preserved by $T$. Then there exists a countable set $\Lambda$ such that $F\subseteq E_\Lambda$. Since $E_\Lambda$ is isomorphic to $\ell_\infty$, by Lemma~\ref{rosl1}(i), $T$ preserves a copy of $\ell_\infty$. \end{proof}\smallskip

\begin{lemma}\label{lemmaideal}Let $\kappa$ and $\lambda$ be infinite cardinal numbers and let $T\in \mathscr{B}(\ell_\infty^c(\lambda))$. Suppose that $T\notin \mathscr{S}_{\ell_\infty^c(\kappa)}(\ell_\infty^c(\lambda))$. Then $\mathscr{S}_{\ell_\infty^c(\kappa^+)}(\ell_\infty^c(\lambda))$ is contained in the ideal generated by $T$.\end{lemma}

\begin{proof}Note that $\mathscr{B}(\ell_\infty^c(\lambda))$ is nothing but $\mathscr{S}_{\ell_\infty^c(\lambda^+)}(\ell_\infty^c(\lambda))$. (Recall that $\lambda^+$ denotes the immediate cardinal successor of $\lambda$.) It is then enough to consider only the case where $\kappa\leqslant \lambda$. For simplicity of notation, for any cardinal number $\kappa$ set $E_\kappa = \ell_\infty^c(\kappa)$.\smallskip

Fix $T\notin \mathscr{S}_{E_\kappa}(E_\lambda)$ and let $E$ be a subspace of $E_\lambda$ isomorphic to $E_\kappa$ on which $T$ is bounded below and such that $F=T[E]$ is complemented (such $E$ exists by Proposition~\ref{stronger}).\smallskip

Let $S\in \mathscr{S}_{E_{\kappa^+}}(E_\lambda)$. By Proposition~\ref{proposition3}, $S$ is not bounded below on any sublattice isometric to $c_0(\kappa^+)$. By Proposition~\ref{proposition2NEW}, for each $\varepsilon > 0$, there exists a subset $\Lambda_\varepsilon \subset \lambda$ such that $|\Lambda_\varepsilon | \leqslant \kappa$ and $\|SR_{\lambda\setminus \Lambda_{\varepsilon}}\|\leqslant \varepsilon$. Putting $\Lambda = \bigcup_{n=1}^\infty \Lambda_{1/n}$, we have $S=SR_\Lambda$. Thus $S$ factors through $\ell_\infty^c(\kappa)$, and hence through $T|_{E}$. Since $F$ is complemented, $S$ factors through $T$ as well.\end{proof}

\section{Proofs of Theorems~\ref{uniquemax}, \ref{listofideals2}, \ref{maincomplemented} and  \ref{listofideals}}
We are now in a position to prove Theorems~\ref{uniquemax}, \ref{listofideals2}, \ref{maincomplemented} and  \ref{listofideals} (we have already proved Theorem~\ref{complemented} in the previous section). 

\begin{proof}[Proof of Theorem~\ref{uniquemax}]The case of $\lambda=\omega$ follows from \cite[Proposition 2.f.4]{lt} as explained in \cite[p.~253]{lauroy}. Suppose then that $\lambda$ is uncountable. The set $\mathscr{S}_X(X)$ is indeed a closed ideal of $\mathscr{B}(X)$ by Theorem~\ref{trueideals}. To show that it is the unique maximal ideal, take $T\notin \mathscr{S}_X(X)$. Then there are subspaces $E$ and $F$ of $X$, both isomorphic to $X$ such that $T_E\colon E\to F$ is an isomorphism and $F$ is complemented (if $X=\ell_\infty(\lambda)$ then it follows directly from injectivity of $\ell_\infty(\lambda)$; if $X=\ell_\infty^c(\lambda)$ then we apply Theorem~\ref{complemented}).

Let $W=(T|_E)^{-1}$ regarded as operator into $X$. Furthermore, let $U\colon X\to F$ be an isomorphism. Since $F$ is complemented, $U^{-1}$ can be extended to an operator on $X$, say $V$. Consequently,
$$I_X =  VT(WU)$$
belongs to the ideal generated by $T$. For this reason, $I_X = ATB$ for some $A,B\in\mathscr{B}(X)$ if and only if $T\notin \mathscr{S}_X(X)$.\end{proof}

\begin{proof}[Proof of Theorem~\ref{listofideals2}]We keep the notation of the proof of Lemma~\ref{lemmaideal}.

Let $\kappa$ be an infinite cardinal. By Lemma~\ref{lemmaideal}, if $T\notin \mathscr{S}_{E_\kappa}(E_\lambda)$, then $\mathscr{S}_{E_{\kappa^+}}(E_\lambda)$ is contained in the ideal generated by $T$. We will show that if $\mathscr{J}$ is a closed ideal of $\mathscr{B}(E_\lambda)$, then $\mathscr{J}=\mathscr{S}_{E_{\kappa}}(E_\lambda)$ for some cardinal $\kappa$. 

Set $$\tau = \sup\{\rho\colon \rho\text{ is a cardinal number and } \mathscr{S}_{E_\rho}(E_\lambda)\subseteq \mathscr{J}\}.$$ 
This number is well defined as, by Lemma~\ref{lemmaW}, $\mathscr{J}$ contains $\mathscr{W}(E_\lambda) = \mathscr{S}_{E_\omega}(E_\lambda)$. If the supremum $\tau$ is attained and $S\in \mathscr{J}\setminus \mathscr{S}_{E_\tau}(E_\lambda)$, then the ideal generated by $S$, hence $\mathscr{J}$ as well, would contain $\mathscr{S}_{E_\tau^+}(E_\lambda)$, which is impossible by the definition of $\tau$. Consequently, $\mathscr{J}= \mathscr{S}_{E_\tau}(E_\lambda)$. Consider now the case where $\tau>\rho$ for each $\rho$ such that $\mathscr{S}_{E_\rho}(E_\lambda)$ is contained in $\mathscr{J}$. Since $\mathscr{J}$ is closed it contains the closure of $\bigcup_{\rho<\kappa} \mathscr{S}_{E_\rho}(E_\lambda)$, which is to say that $\mathscr{S}_{E_\kappa}(E_\lambda)\subseteq \mathscr{J}$ by Lemma~\ref{lemmadaws}. However, if this inclusion were strict $\mathscr{S}_{E_{\kappa^+}}(E_\lambda)$ would be contained in  $\mathscr{J}$, which would again contradict the definition of $\tau$.\end{proof}

\begin{proof}[Proof of Theorem~\ref{maincomplemented}]Suppose that $X$ is an infinite-dimensional complemented subspace of $\ell_\infty^c(\lambda)$ that is not isomorphic to $\ell_\infty$. Then there is a cardinal number $\kappa \leqslant \lambda$ such that $X$ is isomorphic to a complemented subspace of $\ell_\infty^c(\kappa)$ and $c_0(\kappa)$ embeds into $X$. Consequently, Proposition~\ref{proposition3} applies and so $X$ contains a subspace isomorphic to $\ell_\infty^c(\kappa)$. By Theorem~\ref{complemented}, $X$ contains a complemented subspace isomorphic to $\ell_\infty^c(\kappa)$. Since $\ell_\infty^c(\kappa)$ is isomorphic to $\ell_\infty(\ell_\infty^c(\kappa))$, the $\ell_\infty$-sum of countably many copies of itself, the Pe\l czy\'{n}ski decomposition method (\emph{cf.} \cite[Theorem 2.2.3]{ak}) yields that $X$ is isomorphic to $\ell_\infty^c(\kappa)$.\end{proof}

\begin{proof}[Proof of Theorem~\ref{listofideals}]Let $E_\lambda$ denote one of the spaces $c_0(\lambda)$ or $\ell_p(\lambda)$ for some $p\in [1,\infty)$. For $\lambda=\omega$ the result is well-known (see, \emph{e.g}., \cite[Theorem on p.~82]{pietsch}). Suppose then that $\lambda$ is uncountable.

Since $E_\lambda$ has the approximation property, the ideal of compact operators $\mathscr{K}(E_\lambda)$ is the smallest closed non-zero ideal of $\mathscr{B}(E_\lambda)$ and $\mathscr{K}(E_\lambda) = \mathscr{S}_{E_\omega}(E_\lambda)$ (\cite[Theorem on p.~82]{pietsch}). On the other hand $\mathscr{B}(E_\lambda) = \mathscr{S}_{E_{\lambda^+}}(E_\lambda)$. 

Let $\kappa$ be an infinite cardinal. All we need to do is to show that if $T\notin \mathscr{S}_{\ell_\infty^c(\kappa)}(\ell_\infty^c(\lambda))$, then $\mathscr{S}_{\ell_\infty^c(\kappa^+)}(\ell_\infty^c(\lambda))$ is contained in the ideal generated by $T$. However, the remainder of the proof is completely analogous to the proof of Theorem~\ref{listofideals2}, so we leave the details to the reader. \end{proof}

\end{document}